\documentclass{amsart}

\usepackage{enumerate}

\usepackage{amssymb}
\usepackage{graphicx}
\usepackage{MnSymbol}

\usepackage{amsthm}

\title{A note on automatic continuity}
\author{Gregory R. Conner}
\author{Samuel M. Corson}

\theoremstyle{definition}
\theoremstyle{definition}\newtheorem*{B}{Theorem \ref{tonsofexamples}}
\theoremstyle{definition}\newtheorem*{C}{Theorem \ref{countableslender}}

\theoremstyle{definition}\newtheorem{bigtheorem}{Theorem}

\theoremstyle{definition}\newtheorem{theorem}{Theorem}[section]

\theoremstyle{definition}
\theoremstyle{definition}\newtheorem{proposition}[theorem]{Proposition}
\theoremstyle{definition}\newtheorem{definition}[theorem]{Definition}
\theoremstyle{definition}\newtheorem{question}[theorem]{Question}
\theoremstyle{definition}
\theoremstyle{definition}\newtheorem{remark}[theorem]{Remark}
\theoremstyle{definition}
\theoremstyle{definition}\newtheorem{lemma}[theorem]{Lemma}
\theoremstyle{definition}
\theoremstyle{definition}
\theoremstyle{definition}
\theoremstyle{definition}
\theoremstyle{definition}
\newtheorem*{question*}{Question}
\newtheorem*{theorem*}{Theorem}
\newtheorem*{corollary*}{Corollary}
\newtheorem*{lemma*}{Lemma}

\def\pmc#1{\setbox0=\hbox{#1}
    \kern-.1em\copy0\kern-\wd0
    \kern.1em\copy0\kern-\wd0}

\newcommand{\W}{\mathcal{W}}
\newcommand{\HEG}{\textbf{HEG}}

\newcommand{\card}{\operatorname{card}}
\newcommand{\supp}{\operatorname{supp}}
\newcommand{\Ant}{\operatorname{Ant}}
\newcommand{\Frac}{\operatorname{frac}}
\newcommand{\rad}[2]{\sqrt[#1] #2}

\begin{document}

\address{Mathematics Department\\
        Brigham Young University\\
        Provo, UT. 84602\\
        USA}
\email{conner@math.byu.edu}

\address{Ikerbasque- Basque Foundation for Science and Matematika Saila, UPV/EHU, Sarriena S/N, 48940, Leioa - Bizkaia, Spain}

\email{sammyc973@gmail.com}
\keywords{slender group, Thompson's group, word hyperbolic group, completely metrizable group, locally compact Hausdorff group}
\subjclass[2010]{Primary 03E75, 54H11; Secondary 20E06}

\maketitle

\begin{abstract}  We present new results regarding automatic continuity, unifying some diagonalization concepts that have been developed over the years.  For example, any homomorphism from a completely metrizable topological group to Thompson's group $F$ has open kernel.  A similar claim holds when $F$ is replaced with a Baumslag-Solitar group or a torsion-free word hyperbolic group.

\end{abstract}

\begin{section}{Introduction}

Groups in which infinite multiplication can be defined often have difficulty in mapping to groups like free groups and the proof follows a diagonalization argument.  The first such argument was given in the abelian category of groups by Specker in \cite{Sp} and a proof for non-abelian groups was given shortly thereafter by Higman \cite{H}.  Later, Dudley modified these ideas to give some powerful theorems for topological groups in \cite{D}.  More recently such papers as \cite{E}, \cite{CC}, \cite{Na}, and \cite{C1} have used these techniques.  We distill some unifying themes of such arguments and strengthen them.  

The \textit{$j$-th radical} $\rad{j}{S}$ of a subset $S$ of a group $G$ is defined to be $\{g\in G \mid g^j\in S\}$.  A \textit{limiting sequence pair} for a group $G$ is a pair of sequences $F_1 \subseteq  F_2\subseteq \cdots$  and $k_1, k_2, \ldots$ with $F_n \subseteq G$ and $k_n \in \mathbb{N}$ such that for every natural number $n$:

\begin{enumerate}
\item $(\forall g \in G)(\exists m\in \mathbb{N})[gF_n \subseteq F_m]$
\item $\rad{k_n}{F_n} =\{1_G\} $
\item $\rad{k_m}{F_n} \subseteq F_n, ~ \forall m \leq n $

\end{enumerate}

\noindent By condition (2) we get $1_G \in F_n$ for all $n$, and by condition (1) together with $1_G \in F_n$ we get that $\bigcup_{n\in \mathbb{N}} F_n = G$.  As defined in \cite{E} a group $G$ is \emph{n-slender} provided the kernel of any homomorphism from the fundamental group of the Hawaiian earring to $G$ contains the kernel of some retraction to a finite bouquet of circles (see Section \ref{Proofofmaintheorem} for more details).  Alternatively a group $G$ is n-slender if and only if every homomorphism from the fundamental group of a connected, locally path connected metric space to $G$ has open kernel (in the sense of the paper \cite{C2}).  We say a group $G$ is \emph{cm-slender} if the kernel of any homomorphism from a completely metrizable group to $G$ is open.  Similarly $G$ is \emph{lcH-slender} if the kernel of any homomorphism from a locally compact Hausdorff group to $G$ is open.  These concepts can be thought of as strong automatic continuity conditions.  For example, a homomorphism from a completely metrizable group to a cm-slender group must be continuous.

We prove the following in Section \ref{Proofofmaintheorem}:

\begin{bigtheorem} \label{lspslender}  Each group with a limiting sequence pair is n-slender, cm-slender and lcH-slender.
\end{bigtheorem}

This theorem has fairly broad applications, and the form of the argument is one of the only techniques known to the authors for proving slenderness in a non-abelian group (see discussion in Section \ref{conclusion}).  We define some notions for the next theorem.  A group $G$ has \emph{finite roots} if for every $g\in G$ the set $\{h\in G \mid (\exists n\in \mathbb{N})[h^n = g]\}$ is finite.  We say $G$ has \emph{finite k-antecedents} if there exists $k\in \mathbb{N}$ such that for every $g\in G$ the set $\{h\mid(\exists n\in \mathbb{N})[h^{k^n} = g]\}$ is finite.  Certainly finite roots implies finite $k$-antecedents for every $k\in \mathbb{N}$.  Finite groups, for example, have finite roots.  In Section \ref{examples} we prove the following:

\begin{bigtheorem}\label{tonsofexamples}  The following have a limiting sequence pair:

\begin{enumerate}[(a)]

\item Groups with Dudley norm

\item Groups with uniformly monotone length function

\item Countable torsion-free groups with finite $k$-antecedents for some $k\in \mathbb{N}$

\item Countable torsion-free groups with finite roots

\item Free groups

\item Free abelian groups

\item $\mathbb{Z}[\frac{1}{m}]$ for each $m\in \mathbb{N}$

\item Torsion-free word hyperbolic groups

\item Baumslag-Solitar groups

\item Thompson's group $F$

\end{enumerate}

\end{bigtheorem}

By Theorem \ref{lspslender}, all such groups are n-slender, cm-slender and lcH-slender.  Groups (a), (e) and (f) were shown to be cm- and lcH-slender in \cite{D}.  Groups (e) were shown to be n-slender in \cite{H}.  Groups (f) and (g) were shown to be n-slender in \cite{E}.  Groups (i) were shown to be n-slender in \cite{Na}.  Groups (b) and (h) were shown to be n-slender in \cite{C1}.  That (j) is n-slender answers a question asked by the second author in \cite{C1}.

We also give some closure properties of n-, cm- and lcH-slender groups.  While such classes are obviously closed under taking subgroups, it is less obvious that they are closed under direct sums, free products and more generally under graph products (see Theorem \ref{graphproducts}).  That n-slender groups are closed under direct sums and free products was shown by Eda in \cite{E}, and closure under arbitrary graph products was shown more recently in \cite{C1}.  We also give the following:

\begin{bigtheorem}\label{countableslender}  If $A$ is an abelian group of cardinality $<2^{\aleph_0}$ the following are equivalent:

\begin{enumerate}
\item $A$ is slender

\item $A$ is n-slender

\item $A$ is cm-slender

\item $A$ is lcH-slender
\end{enumerate}

\end{bigtheorem}

That the abelian n-slender groups (of any cardinality) are precisely the slender groups was shown in \cite{E}.  We end with some discussion and questions in Section \ref{conclusion}.

\end{section}

\begin{section}{Proof of Theorem \ref{lspslender}}\label{Proofofmaintheorem}

We prove Theorem \ref{lspslender} in the sequence of Propositions \ref{lspnslender}, \ref{lspcmslender} and \ref{lsplcHslender}.  First, we recall the combinatorial description of the Hawaiian earring group.  Recall that the Hawaiian earring is a shrinking wedge of countably many circles.  More precisely, the Hawaiian earring is the space $E=\bigcup_{n\in \mathbb{N}} C((0, \frac{1}{n}),\frac{1}{n}) \subseteq \mathbb{R}^2$ where $C(q, r)$ is the circle centered at $q$ of radius $r$.  Let $\{a_n^{\pm 1}\}_{n\in \mathbb{N}}$ be a countably infinite set such that each element has a formal inverse.  A function $W: \overline{W} \rightarrow \{a_n^{\pm 1}\}_{n\in \mathbb{N}}$ is a \emph{word} if $\overline{W}$ is a countable totally ordered set and for each $N$ the set $W^{-1}(\{a_n^{\pm 1}\}_{n=1}^N)$ is finite.  Identify the words $W_1$ and $W_2$ if there exists an order isomorphism $\iota:\overline{W_1} \rightarrow \overline{W_2}$ such that $W_2(\iota(i)) = W_1(i)$ for all $i\in \overline{W_1}$.  Let $\W$ denote the set of words (under the identification, this is a set with cardinality $2^{\aleph_0}$).  For each $N\in \mathbb{N}$ let $p_N$ be the map from $\W$ to the subset of finite words in $\W$ given by the restriction $p_N(W) = W|\{i\in \overline{W}\mid W(i) \in \{a_n^{\pm 1}\}_{n=1}^{N}\}$.  We write $W_1 \sim W_2$ if for every $N\in \mathbb{N}$ we have $p_N(W_1)$ equal to $p_N(W_2)$ as elements of the free group $F(\{a_n\}_{n=1}^N)$.  

Given words $W_1, W_2\in \W$ we have a word $W_1W_2$ by letting $\overline{W_1W_2}$ be the disjoint union $\overline{W_1} \sqcup \overline{W_2}$ under the order that places all elements of  $\overline{W_1}$ less than those of $\overline{W_2}$ and extends the orders of both $\overline{W_1}$ and $\overline{W_2}$, and $W_1W_2(i) = \begin{cases}W_1(i)$ if $i\in \overline{W_1}\\W_2(i)$ if $i\in \overline{W_2} \end{cases}$.  We form a group $\HEG= \W/\sim$, with binary operation defined by word concatenation: $[W_1][W_2] = [W_1W_2]$.  The identity element is the equivalence class of the word with empty domain.  From a word $W$ we define $W^{-1}$ by letting $\overline{W^{-1}}$ have the reverse order of $\overline{W}$ and $W^{-1}(i) = (W(i))^{-1}$.  This gives the inverse in the group: $[W]^{-1}= [W^{-1}]$.

For each $N\in \mathbb{N}$ the word map $p_N$ defines a retraction homomorphism (which we again denote $p_N$) to the free subgroup $F(\{a_n\}_{n=1}^N)$ in $\HEG$.  Similarly we have a word map $p^N$ defined by the restriction $p^N(W) = W|\{i\in \overline{W}\mid W(i) \notin \{a_n^{\pm 1}\}_{n=1}^{N}\}$.  This word map also defines a retraction $p^N$ to the subgroup of $\HEG$ consisting of those equivalence classes containing a word which involves no elements in $\{a_n^{\pm 1}\}_{n=1}^N$.  Let $\HEG_N$ denote the image of $p_N$ and $\HEG^N$ denote the image of $p^N$.  We have an isomorphism $\HEG \simeq \HEG_N * \HEG^N$ by considering a word as a finite concatenation of words using elements in $\{a_n^{\pm 1}\}_{n=1}^N$ and words excluding such elements.

We are now ready to recall a definition (originally from \cite{E}):

\begin{definition}  A group $G$ is \emph{noncommutatively slender} (or \emph{n-slender}) if for every homomorphism $\phi:\HEG \rightarrow G$ there is an $N\in \mathbb{N}$ such that $\phi = \phi \circ p_N$.  Equivalently, for every homomorphism $\phi:\HEG \rightarrow G$ there exists $N$ such that $\phi|\HEG^N$ is the trivial map.
\end{definition}

Alternatively a group $G$ is n-slender if any homomorphism $\phi$ from the fundamental group of a path connected, locally path connected metric space $Y$ to $G$ has open kernel (see Lemma 6.2 in \cite{C2}).  We proceed to the proof of Theorem \ref{lspslender}.

\begin{lemma}\label{thebiglemma}  Suppose $G$ has limiting sequence pair $(F_n)_n$, $(k_n)_n$ and that $(g_n)_n$ is a sequence in $G$.  There exists a sequence $(m_n)_n$ in $\mathbb{N}$ such that for any $l\in \mathbb{N}$ we have 

\begin{center}  $g_1(g_2(\cdots g_{l-1}(g_lx^{m_l})^{m_{l-1}}      \cdots)^{m_2})^{m_1} \in F_l$
\end{center}
  
\noindent  implies $x = 1_G$.
\end{lemma}

\begin{proof}  We define natural numbers $j_{n, s}$ where $n \geq s \geq 1$.  Pick $j_{1, 1}>1$ such that $g_1^{-1}F_1 \subseteq F_{j_{1, 1}}$ by condition (1) for a limiting sequence.  Pick $j_{2, 1}>j_{1, 1}$ so that $g_1^{-1}F_2 \subseteq F_{j_{2, 1}}$ and $j_{2, 2}>j_{1, 1}$ so that $g_2^{-1}F_{j_{2, 1}}\subseteq F_{j_{2, 2}}$.  Given $n\in \mathbb{N}$ we pick $j_{n, 1}>j_{n-1, n-1}$ so that $g_1^{-1}F_n \subseteq F_{j_{n, 1}}$ and assuming we have selected $j_{n, s}$ for $s<n$ we let $j_{n, s+1}>j_{n, s}$ be such that $g_{s+1}^{-1}F_{j_{n, s}}\subseteq F_{j_{n, s+1}}$.  Let $m_n = k_{j_{n, n}}$.

Suppose for contradiction that for some $l\in \mathbb{N}$ and $x\in G\setminus \{1_G\}$ we have $g_1(g_2(\cdots g_{l-1}(g_lx^{m_l})^{m_{l-1}}      \cdots)^{m_2})^{m_1} \in F_l$.  Then $(g_2(g_3  (\cdots g_{l-1}(g_lx^{m_l})^{m_{l-1}}\cdots)^{m_3})^{m_2})^{m_1} \in g_1^{-1}F_l \subseteq F_{j_{l, 1}}$.  Then $g_2(g_3  (\cdots g_{l-1}(g_lx^{m_l})^{m_{l-1}}\cdots)^{m_3})^{m_2} \in F_{j_{l, 1}}$ by condition (3) for a limiting sequence.  Then $$(g_3  (\cdots g_{l-1}(g_lx^{m_l})^{m_{l-1}}\cdots)^{m_3})^{m_2} \in g_2^{-1}F_{j_{l, 1}}\subseteq F_{j_{1, 2}}$$ and continuing in this way we see that $g_{s}(\cdots (g_lx^{m_l})^{m_{l-1}}\cdots )^{m_s}\in F_{j_{l, s-1}}$, so that eventually $x^{m_l}=x^{k_{j_{l, l}}}\in F_{j_{l, l}}$, contradicting condition (2) for a limiting sequence.
\end{proof}

\begin{proposition}\label{lspnslender}  If $G$ has a limiting sequence pair then $G$ is n-slender.
\end{proposition}

\begin{proof}  Suppose $G$ has limiting sequence pair $(F_n)_n$, $(k_n)_n$.  Suppose for contradiction that $G$ is not n-slender, so that there exists a homomorphism $\phi:\HEG \rightarrow G$ with a sequence $W_n \in \HEG^n$ such that $\phi(W_n) = g_n \neq 1_G$.  Pick the sequence $(m_n)_n$ as in Lemma \ref{thebiglemma}.    For each $i\in \mathbb{N}$ define a word $U_i\in \HEG$ by letting $$U_i = W_{i+1}(W_{i+2}(\cdots W_{i+k}(\cdots)^{m_{i+k}}  \cdots)^{m_{i+2}})^{m_{i+1}}$$ or in other words $U_{i-1} = W_iU_i^{m_i}$, and consequently $$U_0 = W_1(W_2(\cdots)^{m_2})^{m_1}$$  We have $\phi(U_0) \in F_l \subseteq F_{l+1}$ for some $l\in \mathbb{N}$.  But by the conclusion of Lemma \ref{thebiglemma} we know that $\phi(U_l) = \phi(U_{l+1})=1_G$, which implies $ g_l = \phi(W_l) = 1_{G}$, a contradiction.
\end{proof}

\begin{proposition}\label{lspcmslender}  If $G$ has a limiting sequence pair then $G$ is cm-slender.
\end{proposition}

\begin{proof}  Let $G$ have a limiting sequence pair.  Let $H$ be a topological group with complete metric $d$ and $\phi: H\rightarrow G$ a homomorphism.  We show that $\phi^{-1}(1_G)$ must contain a neighborhood of $1_H$.  We define a sequence $(h_n)_n$ in $H$ and sequences $\{j_{n, s}\}_{n \geq s \geq 1}$ and $(m_n)_n$ in $\mathbb{N}$.

Suppose for contradiction that the conclusion fails.  Pick $h_1 \in H\setminus \ker(\phi)$, let $j_{1,1}>1$ be such that $\phi(h_1^{-1})F_1 \subseteq F_{j_{1, 1}}$ and let $m_1 = k_{j_{1, 1}}$.  There is an open neighborhood $V_1$ of $1_H$ so that $h\in V_1$ implies $d(h_1h^{m_1}, h_1) <2^{-1}$.  Pick $h_2\in V_1\setminus \ker(\phi)$.  Pick $j_{2, 1}>j_{1, 1}$ and $j_{2, 2} >j_{2, 1}$ so that $\phi(h_1^{-1})F_2 \subseteq F_{j_{2, 1}}$ and $\phi(h_2^{-1}) F_{j_{2, 1}} \subseteq F_{j_{2, 2}}$.  Let $m_2 = k_{j_{2, 2}}$.  There is an open neighborhood $V_2$ of $1_H$ such that for any $h\in V_2$ we have both

\begin{center} $d(h_1(h_2(h)^{m_2})^{m_1}, h_1h_2^{m_{1}})<2^{-2}$
\end{center}

\noindent and

\begin{center}  $d(h_2, h_2h^{m_2})< 2^{-2}$
\end{center}

Supposing $h_1, \ldots, h_{n}$, $j_{q, s}$ for $n\geq q \geq s \geq 1$ and $m_1,\ldots, m_{n}$ and neighborhoods $V_1, \ldots, V_n$ have been chosen in this way, we select $h_{n+1} \in V_n \setminus \ker(\phi)$.  Select $j_{n+1, s}$ with $1\leq s\leq n+1$ such that $j_{n, n}<j_{n+1, 1}< j_{n+1, 2}< \cdots< j_{n+1, n+1}$ so that $\phi(h_1^{-1})F_{n+1}\subseteq F_{j_{n+1, 1}}$, $\phi(h_2^{-1})F_{j_{n+1, 1}} \subseteq F_{j_{n+1, 2}}, \cdots,$ $\phi(h_{n+1}^{-1})F_{j_{n+1, n}}\subseteq F_{j_{n+1, n+1}}$.  Let $m_{n+1} = k_{j_{n+1, n+1}}$.  Select a neighborhood of identity $V_{n+1}$ so that $h\in V_{n+1}$ implies 

\begin{center}  $d(h_1(\cdots  (h_{n+1}h^{m_{n+1}})^{m_{n}}\cdots )^{m_1}, h_1(\cdots h_{n+1}^{m_n}  \cdots)^{m_1})<2^{-n-1}$

$d(h_2(\cdots  (h_{n+1}h^{m_{n+1}})^{m_{n}}\cdots )^{m_2}, h_2(\cdots h_{n+1}^{m_n}  \cdots)^{m_2})<2^{-n-1}$

$\vdots$

$d(h_{n+1}h^{m_{n+1}}, h_{n+1})<2^{-n-1}$
\end{center}

Fixing $n \in \mathbb{N}$ we have that $r \geq n$ implies

\begin{center}
$d(h_n(h_{n+1}(\cdots  h_r^{m_{r-1}}\cdots)^{m_{n+1}})^{m_{n}}, h_n(h_{n+1}(\cdots  (h_rh_{r+1}^{m_{r}})^{m_{r-1}}\cdots)^{m_{n+1}})^{m_{n}})<2^{-r-1}$
\end{center}

\noindent and so the sequence is Cauchy in $r$ and we let $y_n$ be the limit.  We have the relationships $y_n = h_n(y_{n+1})^{m_n}$ for all $n\in \mathbb{N}$.  We know $\phi(y_1) \in F_l$ for some $l\in \mathbb{N}$.  Notice also that by construction $1< j_{1, 1}<j_{2, 2}<\cdots$.  As $y_1 = h_1(h_2(\cdots    h_l(y_{l+1})^{m_l}\cdots)^{m_2})^{m_1}$, we know $\phi((h_2(\cdots    h_l(y_{l+1})^{m_l}\cdots)^{m_2}))^{m_1} \in \phi(h_1^{-1})F_l \subseteq F_{j_{l, 1}}$, so that $\phi(h_2(\cdots    h_l(y_{l+1})^{m_l}\cdots)^{m_2}) \in F_{j_{l, 1}}$.  Then $\phi((h_3(\cdots  h_l(y_{l+1})^{m_l}\cdots)^{m_3})^{m_2}) \in \phi(h_2^{-1})F_{j_{l, 2}}$.  Eventually we see that $\phi(y_{l+1}^{m_l}) \in \phi(h_l^{-1})F_{j_{l, l-1}}\subseteq F_{j_{l, l}}$, and since $m_l = k_{j_{l, l}}$ we have $\phi(y_{l+1}) = 1_G$.  Since $\phi(y_1) \in F_l \subseteq F_{l+1}$ the comparable proof reveals that $\phi(y_{l+2}) = 1_G$ as well.  Then $1_G = \phi(y_{l+1}) = \phi(h_{l+1})(\phi(y_{l+2}))^{m_{l+1}} = \phi(h_{l+1})$, contradicting our choice of $h_{l+1}\in V_l \setminus \ker(\phi)$.
\end{proof}

\begin{proposition}\label{lsplcHslender}  If $G$ has a limiting sequence pair then $G$ is lcH-slender.
\end{proposition}

\begin{proof}  Let $\phi:H \rightarrow G$ be a homomorphism with $H$ locally compact Hausdorff and $G$ having a limiting sequence pair.  Suppose for contradiction that $\ker(\phi)$ does not contain an open neighborhood of $1_H$.  Pick an open neighborhood $V_0$ of $1_H$ with $\overline{V_0}$ compact.  Pick $h_1\in H\setminus \ker(\phi)$, pick $j_{1, 1}>1$ such that $\phi(h_1^{-1})F_1 \subseteq F_{j_{1, 1}}$ and let $m_1 = k_{j_{1, 1}}$.  Pick $V_1\subseteq V_0$ such that $h\in V_1$ implies $h^{m_1}\in V_0$.  Pick $h_2\in V_1 \setminus \ker(\phi)$.  Select $j_{2, 1}$ and $j_{2, 2}$ so that $j_{1, 1}< j_{2, 1}<j_{2, 2}$ and $\phi(h_1^{-1})F_2 \subseteq F_{j_{2, 1}}$ and $\phi(h_2^{-1})F_{j_{2, 1}} \subseteq F_{j_{2, 2}}$.  Let $m_2 = j_{2, 2}$ and select $V_2\subseteq V_1$ so that $h\in V_2$ implies $h_2h^{m_2}\in V_1$.

Continue in this way, selecting $j_{n+1, s}$ for $1\leq s\leq n+1$ with $j_{n, n}<j_{n+1, 1}<\cdots <j_{n+1, n+1}$ such that $\phi(h_1^{-1}) F_{n+1}\subseteq F_{j_{n+1, 1}}, \phi(h_2^{-1})F_{j_{n+1, 1}} \subseteq F_{j_{n+1, 2}}, \ldots, \phi(h_{n+1}^{-1})F_{j_{n+1, n}}\subseteq F_{j_{n+1, n+1}}$.   Let $m_{n+1} = k_{j_{n+1, n+1}}$.  Select $V_{n+1}$ so that $h\in V_{n+1}$ implies $h_{n+1}h^{m_{n+1}}\in V_n$.  Define $K_n = h_1(h_2(\cdots h_{n-1}(h_n\overline{V_n}^{m_n}  )^{m_{n-1}}\cdots)^{m_2})^{m_1}$.  We have $K_1 \supseteq K_2 \supseteq \cdots$, and all $K_n$ are non-empty compact, thus $\bigcap_{n\in \mathbb{N}}K_n$ is nonempty.  Let $y \in \bigcap_{n}K_n$.  We have $\phi(y) \in F_l$ for some $l\in \mathbb{N}$.  As $y\in K_{l+1}$ there exists an element $h\in \overline{V_{l+1}}$ such that $y = h_1(h_2(\cdots h_{l+1}h^{m_{l+1}}  \cdots)^{m_2})^{m_1}$.  Letting $h' = h_{l+1}h^{m_{l+1}}$ we can argue as before that $\phi(h') = 1_G$ and that $\phi(h) = 1_G$, whence $\phi(h_{l+1}) = 1_G$, contrary to our choice of $h_{l+1}\notin\ker(\phi)$.
\end{proof}

\end{section}

\begin{section}{Proof of Theorems \ref{tonsofexamples} and \ref{countableslender}}\label{examples}

We recall some definitions in anticipation of proving Theorem \ref{tonsofexamples}.  Recall that a \emph{length function} on a group $G$ is a function $l:G\rightarrow [0, \infty)$ satisfying

\begin{enumerate}[(a)]

\item $l(1_G) = 0$

\item $l(g) = l(g^{-1})$

\item $l(gh) \leq l(g) + l(h)$

\end{enumerate}

\noindent We shall say a length function is a \emph{Dudley norm} if it takes values in $\mathbb{N}\cup\{0\}$ and satisfies $l(g^n) \geq \max\{n, l(g)\}$ whenever $g \neq 1_G$ (this is the norm described in \cite{D}).  A length function $l$ is \emph{uniformly monotone} if there exists $k\in \mathbb{N}$ such that for all $g\neq 1_G$ we have $l(g^k) \geq l(g)+1$ (see \cite{C1}).  

For a review of the Baumslag-Solitar groups the reader may go to Appendix \ref{BSgroups}.  Recall that Thompson's group $F$ is the set of piecewise linear self-homeomorphisms on $[0, 1]$ which are non-differentiable at finitely many points and have derivatives that are powers of $2$ at points of differentiability (see \cite{CFP} for an exposition).  We restate Theorem \ref{tonsofexamples} and provide the proof.

\begin{B}  The following have a limiting sequence pair:

\begin{enumerate}[(a)]
\item Groups with Dudley norm

\item Groups with uniformly monotone length function

\item Countable torsion-free groups with finite $k$-antecedents for some $k\in \mathbb{N}$

\item Countable torsion-free groups with finite roots

\item Free groups

\item Free abelian groups

\item $\mathbb{Z}[\frac{1}{m}]$ for each $m\in \mathbb{N}$

\item Torsion-free word hyperbolic groups

\item Baumslag-Solitar groups

\item Thompson's group $F$
\end{enumerate}

\end{B}

\begin{proof}  For (a) we let $F_n$ be the closed ball $B(1_G, n) = \{g\in G\mid l(g) \leq n\}$ and let $k_n = n+1$.  Conditions (1) - (3) are straightforward to check.  For (b) we similarly let $F_n = B(1_G, n)$ and $k_n = k^n$, where $k$ is the natural number associated with the uniform monotonicity.

For (c) we let $\Ant_k(g) = \{h\in G\mid(\exists n\geq 0)[h^{k^n} = g]\}$.  Let $G = \{g_1 = 1_G, g_2, \ldots\}$ be an enumeration of $G$ and let $F_n = \Ant_k\{g_1, \ldots, g_n\}$.  For each $n\in \mathbb{N}$ pick $k_n$ to be a large enough power $k^m$ of $k$ so that for $j \geq m$ we have $x, x^{k^j}\in F_n$ implies $x = 1_G$ (here we are using the fact that $G$ is torsion-free and $F_n$ is finite).  That $\bigcup_{n\in\mathbb{N}} F_n = G$ holds is clear.  As each $F_n$ is finite we automatically get condition (1) from $\bigcup_{n\in\mathbb{N}} F_n = G$.  Condition (2) follows from our choice of $k_n$.  For condition (3), if $g\notin F_m$ then no $g^{k^j}$ is in $F_m$ for any $j\in \mathbb{N}$ since $F_m$ is closed under taking $k$-antecedents.  The groups (d) are a subclass of (c).

For (e) and (f) we may apply any one of (a), (b) or (c).  For (g) we pick a prime $p>m$ and note that $\mathbb{Z}[
\frac{1}{m}]$ has finite $p$-antecedents and is countable torsion-free, and apply (c).

For (h) we note that torsion-free word hyperbolic groups are certainly countable.  Such groups also have unique root extraction (see Proposition 2.16 of \cite{BV}).  In a torsion-free word hyperbolic group the centralizer of every element is isomorphic to $\mathbb{Z}$ (see again Proposition 2.16 in \cite{BV}), so in such a group there are finite roots and we are done by (d).  Alternatively, torsion-free word hyperbolic groups have a uniformly monotone length function (see \cite{C1}) and so we can apply part (b).

For (i) we fix $m, n \in \mathbb{Z} \setminus\{0\}$ and note that $BS(m, n)$ has unique $p$-th root extraction for any prime $p>|m|, |n|$ (see Appendix A Proposition \ref{BSlemma}).  By Theorem 1 of \cite{Ne} we know that if $p$ is a prime greater than the length of the relator for a one-relator group then elements of that group are not divisible by more than finitely many powers of $p$.  Thus for any prime $p>|m| + |n| +2$ we get finite $p$-antecedents and we are done by (c).

For (j), the group $F$ is torsion-free and has finite roots (see Lemma 15.29 of \cite{GS}).  The group $F$ is finitely presented and therefore countable, and so we apply (d).
\end{proof}

\begin{remark}  Not every countable abelian n-slender group has a limiting sequence pair where the $F_n$ are finite.  Consider the abelian group $G = \bigoplus_{m\in \mathbb{N}}\mathbb{Z}[\frac{1}{m}]$.  Were $(F_n)_n$, $(k_n)_n$ to be a limiting sequence pair, we select $m\in \mathbb{N}$ large enough that $F_m \cap \mathbb{Z}[\frac{1}{k_1}] \neq \{0\}$.  Picking $g\in F_m \cap \mathbb{Z}[\frac{1}{k_1}] \setminus\{0\}$ we have by condition (3) that all the $k_1^r$ roots of $g$ are in $F_m$, and thus we get infinitely many elements in $F_m$.  It is not clear whether $G$ has a limiting sequence pair.  However, $G$ is n-, cm- and lcH-slender by Theorem \ref{graphproducts} or by Theorem \ref{countableslender}.
\end{remark}

\begin{lemma}\label{-by-}  An extension of an n-slender, cm-slender, or lcH-slender group by a group of the same type is again such a group.
\end{lemma}

\begin{proof}  That an n-slender-by-n-slender group is again n-slender is known (Lemma 10 in \cite{C1}, for example).  Suppose we have a short exact sequence $1\rightarrow K \rightarrow^{\iota} G\rightarrow^{q} Q\rightarrow 1$ with both $K$ and $Q$ cm-slender and let $\phi:H\rightarrow G$ be a homomorphism from a completely metrizable group $H$.  As the homomorphism $q\circ \phi$ is to a cm-slender group, we know that the kernel is a clopen subgroup in $H$ and that $\phi(\ker(q\circ \phi))\leq \ker(q) = K$.  Since $\ker(q\circ\phi)$ is itself a completely metrizable group and $K$ is cm-slender we know $\ker(\phi) =  \ker (\phi|\ker(q\circ \phi))$ is a clopen subgroup of $\ker(q\circ \phi)$ and therefore also a clopen subgroup of $H$.  The proof in case $H$ is locally compact Hausdorff is completely analogous.
\end{proof}

\begin{theorem}\label{graphproducts}  The classes of n-slender, cm-slender and lcH-slender groups are each closed under taking graph products.  In particular, each class is closed under free products and direct sums.
\end{theorem}

\begin{proof}  It was shown in \cite{C1} that the n-slender groups are closed under graph products, and we use essentially the same argument here for the other two classes.  In \cite{C1} it was shown that the kernel of the map $\sigma: \Gamma(V, \{G_v\}_{v\in V(\Gamma)}) \rightarrow \bigoplus_{v\in V}$ from the graph product to the direct sum is uniformly monotone.  Thus by Lemma \ref{-by-} it suffices to show closure of the remaining two classes under direct sum.  For this we give an argument analogous to that of Theorem 3.6 in \cite{E}.

Let $G = \bigoplus_{i\in I}G_i$ be a direct sum of cm-slender groups and let $\phi:H \rightarrow G$ be a homomorphism with $H$ a topological group completely metrized by $d$.  In case the conclusion is false we select $h_1 \in H \setminus \ker(\phi)$.  For each $i\in I$ let $p_i: G \rightarrow G_i$ be the projection map.  Given $g\in G$ we let $\supp(g) = \{i\in I\mid p_i(g) \neq 1\}$ denote the support of $g$.  As the $G_i$ are all cm-slender, there exists $\epsilon_1>0$ such that $h\in B(1_H, \epsilon_1)$ implies $p_i\circ \phi(h) = 1$ for all $i\in \supp(\phi(h_1))$.  Pick a neighborhood $V_1$ of $1_H$ such that $h\in V_1$ implies

\begin{center}  $d(1_H, h) <\frac{\epsilon_1}{4}$

$d(h_1, h_1h)< \frac{\epsilon_1}{4}$
\end{center}

Select $h_2\in V_1 \setminus \ker(\phi)$.  Now there exists $\frac{\epsilon_1}{2}>\epsilon_2>0$ such that $h\in B(1_H, \epsilon_2)$ implies $p_i\circ \phi(h) =1$ for all $i\in \supp(\phi(h_2))$.  Pick a neighborhood $V_2$ of $1_H$ such that $h\in V_2$ implies

\begin{center}  $d(1_H, h) < \frac{\epsilon_2}{4}$

$d(h_2, h_2h)<\frac{\epsilon_2}{4}$

$d(h_1h_2, h_1h_2 h) <\frac{\epsilon_2}{4}$
\end{center}

Select $h_3\in V_2$.  Continuing to select $V_n$, $\epsilon_n$, and $h_n$ in this manner we get that for each $n\in \mathbb{N}$ the sequence $(h_n\cdots h_q)_q$ is Cauchy in the variable $q \geq n$.  Letting $y_n = \lim_{q\rightarrow \infty} h_n\cdots h_q$ we see in fact that $d(1_H, y_{n+1})\leq \frac{\epsilon_n}{2}<\epsilon_n$ and the relations $y_m = h_m\cdots h_ny_{n+1}$ hold for $m\leq n$.  Also, the supports $\supp(\phi(h_n))$ are pairwise disjoint and $\bigcup_{m=1}^{n}\supp(\phi(h_m))$ is disjoint from $\supp(\phi(y_{n+1}))$ by the selection of $\epsilon_n$.

As $y_1 = h_1\cdots h_my_{m+1}$ we have $\supp(\phi(y_{m+1})) \subseteq \supp(\phi(h_1))\cup\cdots\cup\supp(\phi(h_m))\cup \supp(\phi(y_1))$, so that in fact $\supp(\phi(y_{m+1})) \subseteq \supp(\phi(y_1))$ for all $m\in \mathbb{N}$.    As $\supp(\phi(y_1))$ is finite we can select $j\in \mathbb{N}$ such that $\supp(\phi(y_1))\cap(\bigcup_{n\geq j}\supp(\phi(h_n)))=\emptyset$.  Now for $m\geq j$ we know $\supp(\phi(y_{m+1}))$ is disjoint from $\bigcup_{n\in \mathbb{N}}\supp(\phi(h_n))$ as $\supp(\phi(y_{m+1}))$ is disjoint from $\bigcup_{n\leq m}\supp(\phi(h_n))$ as well as from $\bigcup_{n> m}\supp(\phi(h_n))$.  But $\supp(\phi(h_{j+1})) \subseteq \supp(\phi(y_{j+1}))\cup \supp(\phi(y_{j+2}))$, so in fact $\supp(\phi(h_{j+1})) = \emptyset$ and $\phi(h_{j+1})=1_G$, contrary to our selection.

The proof for $H$ a locally compact Hausdorff group is similar.  Let $V_0$ be an open neighborhood of $1_H$ with $\overline{V_0}$ compact.  Pick $h_1\in V_0 \setminus \ker(\phi)$.  Select a neighborhood $V_1$ of $1_H$ with $\overline{V_1} \subseteq V_0$ such that $h\in V_1$ implies $p_i(\phi(h))=1_G$ for all $i\in \supp(\phi(h_1))$ and $h_1h\in V_0$.  Pick a neighborhood $V_1'$ of $1_H$ such that $\overline{V_1'} \subseteq V_1$.  Pick $h_2\in V_1' \setminus \ker(\phi)$.  Let $V_2$ be an open neighborhood of $1_H$ with $\overline{V_2}\subseteq V_1'$ such that $h\in V_2$ implies $p_i(\phi(h))=1$ for all $i\in \supp(\phi(h_2))$ and $h_2h\in V_1'$.  Continue selecting $h_n,  V_n$, and $V_n'$ in this way.  Let $K_n = h_1\cdots h_n\overline{V_n'}$ and select $y\in \bigcap_{n\in \mathbb{N}}K_n$.  Pick $j\in \mathbb{N}$ such that $\supp(\phi(y))$ is disjoint from $\bigcup_{n\geq j}\supp(\phi(h_n))$.  Pick $h\in \overline{V_{j+2}'}$ such that $y = h_1\cdots h_{j+2}h$ and let $h' = h_{j+2}h$.  We have $h' = h_{j+2}h\in V_{j+2}'\subseteq V_{j+1}$ and $h\in \overline{V_{j+2}'}\subseteq V_{j+1}$.  As before we can argue that $\phi(h_{j+2}) = 1_G$, a contradiction.
\end{proof}

\begin{remark}  Notice that no cm-slender or lcH-slender groups can have torsion or contain $\mathbb{Q}$ as a subgroup.  If $\mathbb{Q} \leq G$ then there is a discontinuous homomorphism from $\mathbb{R}$ to $\mathbb{Q}$ (constructed by taking a Hamel basis of $\mathbb{R}$ over $\mathbb{Q}$).  If $G$ contains a torsion element, then $G$ contains an element of prime order, so in particular $\mathbb{Z}/p\mathbb{Z} \leq G$.  The product $\prod_{\mathbb{N}}\mathbb{Z}/p\mathbb{Z}$ is a compact Polish group and by using a vector space argument we may again construct a discontinuous homomorphism from the product to $\mathbb{Z}/p\mathbb{Z}$.  
\end{remark}

Recall that an abelian group $A$ is \emph{slender} if for each homomorphism $\phi: \prod_{\mathbb{N}} \mathbb{Z} \rightarrow A$ there exists a natural number $n$ such that $\phi = \phi \circ p_n$ where $p_n:  \prod_{\mathbb{N}}\mathbb{Z} \rightarrow \bigoplus_{i=1}^n\mathbb{Z}$ is projection to the first $n$ coordinates.  It is known that the abelian n-slender groups are precisely the slender groups (see Theorem 3.3 of \cite{E}).  One could ask how the abelian cm-slender and lcH-slender groups compare with the slender groups.  Since $\prod_{\mathbb{N}}\mathbb{Z}$ is a Polish group we easily have that an abelian cm-slender group must certainly be slender.  Theorem \ref{countableslender} provides a partial answer.  Towards this theorem we prove the following proposition which is of interest in its own right.

\begin{proposition}\label{artinian}  If $H$ is either completely metrizable or locally compact Hausdorff, $\phi:H\rightarrow G$ is a group homomorphism and $\card(G) < 2^{\aleph_0}$ then there exists an open neighborhood $V$ of $1_H$ such that for any neighborhood $1_H \in V' \subseteq V$ we have $\phi(V') = \phi(V)$.
\end{proposition}

\begin{proof}  We first prove the claim in case $H$ is completely metrizable.  Let $d$ be a complete metric which induces the topology on $H$.  Supposing the claim is false we get a sequence $(n_k)_k$ of natural numbers such that:

\begin{enumerate}
\item $n_{k+1} >2n_k$

\item $\phi(B(1_H, \frac{1}{n_k}))\supsetneq \phi(B(1_H, \frac{1}{n_{k+1}})B(1_H, \frac{1}{n_{k+1}})^{-1})$
\end{enumerate}

Again, we are using closed balls $B(1_H, \lambda) = \{x\in H\mid d(1_H, x) \leq \lambda\}$.  Let $D_k = B(1_H, \frac{1}{n_k})$.  Define a subsequence $(k_m)_m$ and sequence $(h_m)_m$ in $H$ as follows.  Let $k_1 = 1$ and select $h_1\in D_1$ such that $\phi(h_1)\notin \phi(D_2D_2^{-1})$.  Pick $k_2> k_1+1$ such that $h\in D_{k_2}$ implies

\begin{center}  $d(h_1, h_1h)<\frac{1}{2n_{k_1+1}}$

$d(1_H, h)<\frac{1}{2n_{k_1+1}}$
\end{center}

Select $h_2\in D_{k_2}$ such that $\phi(h_2)\notin \phi(D_{k_2 +1}D_{k_2 + 1}^{-1})$.  Select $k_3>k_2+1$ such that $h\in D_{k_3}$ implies $d(h_1^{\epsilon_1}h_2^{\epsilon_2}, h_1^{\epsilon_1}h_2^{\epsilon_2}h)<\frac{1}{2n_{k_2+1}}$ for any $\epsilon_1, \epsilon_2\in \{0, 1\}$.  If one has selected $h_1, \ldots, h_m$ and $k_1, \ldots, k_m$ we select $k_{m+1}>k_m+1$ such that $h\in D_{k_{m+1}}$ implies $d(h_1^{\epsilon_1}\cdots h_m^{\epsilon_m}, h_1^{\epsilon_1}\cdots h_m^{\epsilon_m}h)<\frac{1}{2n_{k_{m}+1}}$ for any sequence $(\epsilon_1, \ldots, \epsilon_m)$ in $0$s and $1$s.  Select $h_{m+1}\in D_{k_{m+1}}$ such that $\phi(h_{m+1})\notin \phi(D_{k_{m+1}+1}D_{k_{m+1}+1}^{-1})$.

By construction, for every sequence $\epsilon\in \{0, 1\}^{\mathbb{N}}$ the sequence $(h_1^{\epsilon_1}\cdots h_m^{\epsilon_m})_m$ is Cauchy and converges to, say, $h_{\epsilon}$.  Also, for every $m>1$ and $\epsilon = (\epsilon_1, \ldots)$ we have $h_{\epsilon}\in h^{\epsilon_1}\cdots h_m^{\epsilon_m}D_{k_m +1}$.  By construction we know $$\phi(h_1^{\epsilon_1}\cdots h_{m-1}^{\epsilon_{m-1}}h_m^0D_{k_m+1}) \cap \phi(h_1^{\epsilon_1}\cdots h_{m-1}^{\epsilon_{m-1}}h_m^1D_{k_m+1}) = \emptyset$$  Thus for distinct $\epsilon, \epsilon'\in \{0, 1\}^{\mathbb{N}}$ we have $\phi(h_{\epsilon})\neq\phi(h_{\epsilon'})$, contradicting our assumption on the cardinality of $G$.

The proof for $H$ locally compact Hausdorff is similar.  Let $V_0$ be an open neighborhood of $1_H$ such that $\overline{V_0}$ is compact.  Pick a sequence $V_0 \supseteq V_1 \supseteq V_1'\supseteq V_2 \supseteq V_2' \supseteq V_3$ of open neighborhoods of $1_H$ such that $\overline{V_{m+1}}\subseteq V_m'$ and $\phi(V_n)\supsetneq \phi(V_n'(V_n')^{-1})$.  Select $h_n\in V_n$ such that $\phi(h_n) \notin \phi(V_n'(V_n')^{-1})$.  For each $\epsilon \in \{0, 1\}^{\mathbb{N}}$ and $m\in \mathbb{N}$ we let $K_{\epsilon, m} = h_1^{\epsilon_1}\cdots h_m^{\epsilon_m}\overline{V_{m+1}}$.  Select $h_{\epsilon}\in \bigcap_{m\in \mathbb{N}}K_{\epsilon, m}$.  Notice that $h_{\epsilon}\in h_1^{\epsilon_1}\cdots h_m^{\epsilon_m}\overline{V_{m+1}}\subseteq  h_1^{\epsilon_1}\cdots h_m^{\epsilon_m}V_m'$.  Arguing as before we get $\phi(h_{\epsilon})\neq \phi(h_{\epsilon'})$ for $\epsilon\neq \epsilon'$.
\end{proof}

For $V$ satisfying the conclusion of Proposition \ref{artinian} the image $\phi(V)$ is a subgroup of $G$ regardless of whether $V$ was itself a subgroup of $H$, which is not obvious a priori.  For the following theorem we use the fact from \cite{Sa} that a torsion-free abelian group of cardinality $<2^{\aleph_0}$ is slender if and only if it is reduced (i.e. if and only if $\bigcap_{m\in \mathbb{N}}mA = \{0\}$, or if and only if $\mathbb{Q}$ is not a subgroup).

\begin{C}  If $A$ is an abelian group of cardinality $<2^{\aleph_0}$ the following are equivalent:

\begin{enumerate}
\item $A$ is slender

\item $A$ is n-slender

\item $A$ is cm-slender

\item $A$ is lcH-slender
\end{enumerate}

\end{C}

\begin{proof}  What needs to be shown is that if $H$ is completely metrizable or locally compact Hausdorff and $A$ is abelian, torsion-free, reduced and of cardinality $<2^{\aleph_0}$ then any homomorphism $\phi:H\rightarrow A$ has open kernel.  To avoid confusion we use multiplicative notation with the abelian group $A$.  Thus $1_A$ is identity and $A^j$ is the subgroup consisting of those elements of $A$ which have a $j$ root.  Supposing that $H$ is completely metrizable, by Proposition \ref{artinian} we select a neighborhood $V$ of $1_H$ such that any open neighborhood of $1_H$ inside $V$ has image $\phi(V)$ under $\phi$.  If $a\in \phi(V)\setminus \{1_A\}$, we have that $a\notin A^m$ for some $m\in \mathbb{N}$.  Since $A$ is torsion-free we have unique root extraction.  Thus $a^{r} \notin A^{rm}$ for every $r\in \mathbb{N}$.  Pick a sequence $(h_n)_n$ so that $\phi(h_n) = a^n$ and so that $h_{\epsilon} = h_1^{\epsilon_1}(h_2^{\epsilon_2}(h_3^{\epsilon_3}(\cdots)^{3m})^{2m})^{m}$ exists for any $\epsilon\in \{0,1\}^{\mathbb{N}}$.  As $\card(G)<2^{\aleph_0}$ there exist distinct $\epsilon, \epsilon'\in \{0,1\}^{\mathbb{N}}$ such that $\phi(h_\epsilon) = \phi(h_{\epsilon'})$.  Let $n\in \mathbb{N}$ be the smallest such that $\epsilon_n \neq \epsilon_n'$.  Modifying $\epsilon$ and $\epsilon'$ to have $0$ entries up to the $n-1$ place we get again $\phi(h_{\epsilon}) = \phi(h_{\epsilon'})$.  If without loss of generality $\epsilon_n = 1$ we let $\epsilon''$ be equal to $\epsilon$ except with a $0$ at position $n$.  Thus $h_nh_{\epsilon''} = h_{\epsilon}$ and 

\begin{center}
$a^{n} = \phi(h_n) = \phi(h_{\epsilon}h_{\epsilon''}^{-1})= \phi(h_{\epsilon'}h_{\epsilon''}^{-1})\in A^{nm}$
\end{center}

\noindent which is a contradiction.

The proof for $H$ a locally compact Hausdorff group is similar.  If one wishes, a similar argument shows that a torsion-free reduced abelian group of cardinality $<2^{\aleph_0}$ is n-slender by using Theorem 4.4 (1) of \cite{CC}.  
\end{proof}
\end{section}

\begin{section}{Conclusion}\label{conclusion}
We end with some discussion and questions.  First we ask about a possible approach to generalizing part (i) of Theorem  \ref{tonsofexamples}.

\begin{question}  If $G$ is a torsion-free one-relator group and $p$ is a prime greater than the length of the relator then does $G$ have unique $p$-roots? 
\end{question}

\noindent If one has an affirmative answer to this question, then by the same proof as Theorem \ref{tonsofexamples} part (i) we get that torsion-free one-relator groups have a limiting sequence pair.  Thus such groups would be n-slender, which was conjectured in \cite{Na}.

For the next question we recall that a group element $g\in G\setminus\{1_G\}$ is \emph{divisible} if $g$ has an $n$-root in $G$ for every $n\in \mathbb{N}$.

\begin{question} Does a countable torsion-free group with no divisible elements have a limiting sequence pair?
\end{question}

We finally remark that, as far as the authors are aware, the following are the known techniques for determining that a nonabelian group is slender (in the n-, cm-, or lcH- sense):

\begin{itemize}\item Express the group as a slender extension of a slender group (as in Lemma \ref{-by-}).

\item Algebraically decompose the group as a graph product of slender groups (as in Theorem \ref{graphproducts}).

\item Algebraically decompose the group as a certain type of HNN-extension or amalgamated free product (as in \cite{Na}).

\item Perform a diagonalization argument as in the proof of Theorem \ref{lspslender} by constructing an element in the domain of form $h_1(h_2(h_3(\cdots )^{m_3})^{m_2})^{m_1}$ (as is also done in \cite{H} and \cite{D}).  
\end{itemize}

\begin{question}  What other forms of argument can be used to demonstrate these types of slenderness in a nonabelian group?

\end{question}

\end{section}

\appendix

\begin{section}{Baumslag-Solitar Groups}\label{BSgroups}

Recall that the Baumslag-Solitar groups are the groups of form $BS(m, n) = \langle a, t| t^{-1}a^m t = a^n\rangle$ where $m, n \in \mathbb{Z}\setminus \{0\}$.  We review some of the properties of words in the letters $\{a^{\pm1}, t^{\pm 1}\}$ before stating and proving Lemma \ref{BSlemma}, which was used in Theorem \ref{tonsofexamples}.  The elements $a, t$ obviously generate $G = BS(m, n)$.  If $w, w'$ are words in the letters $\{a^{\pm1}, t^{\pm 1}\}$ we'll write $w \equiv w'$ if $w$ is the same word as $w'$ and we'll write $w = _G w'$ if $w$ represents the same element in $G$ as does $w'$, and $w =_G g$ for $g\in G$ if $w$ represents the group element $g$.  Given a word $w\in \{a^{\pm1}, t^{\pm 1}\}^*$ we may repeatedly perform free reductions and obtain a word $w'$ which has no subwords $a^{\pm 1}a^{\mp 1}$ or $t^{\pm 1}t^{\mp 1}$ which represents the same element of the free group $F(a, t)$ and therefore $w =_G w'$ a fortiori.  Such a freely reduced word may be written as $w' \equiv a^{k_0}t^{\epsilon_1}a^{k_1}\cdots t^{\epsilon_j}a^{k_j}$ where $k_i \in \mathbb{Z}$, $\epsilon_i\in \{\pm 1\}$, and $[\epsilon_i = \pm1 \wedge k_i=0]\Rightarrow \epsilon_{i+1} = \pm 1$.

A freely reduced word $w \equiv a^{k_0}t^{\epsilon_1}a^{k_1}\cdots t^{\epsilon_j}a^{k_j}$ is said to be \emph{reduced} if in addition $w$ has no subword of form $t^{-1}a^{k_i}t$ with $k_i\in m\mathbb{Z}$ or of form $t a^{k_i}t^{-1}$ with $k_i \in n\mathbb{Z}$.  From a freely reduced word $w$ one can obtain a reduced word $w'$ with $w =_G w'$ by replacing subwords of form $t^{-1}a^{k_i}t$ with $k_i\in m\mathbb{Z}$ with $a^{k_i\frac{n}{m}}$ (or the other appropriate replacement), freely reducing, and repeating until this is no longer possible and we obtain such a word $w'$.  The number of letters $t^{\pm 1}$ in the word $w$ divided by two gives an upper bound on the number of times this replacement process can be iterated.

By Britton's Lemma (see \cite[Section IV.2]{LS}), for any nonempty reduced word $w$ we have $w \neq_G 1_G$.  From this one sees that if $w = a^{k_0}t^{\epsilon_1}a^{k_1}t^{\epsilon_2}a^{k_2}\cdots a^{k_{j-1}}t^{\epsilon_j}a^{k_j}$ and $w' = a^{l_0}t^{\delta_1}a^{l_2}\cdots t^{\delta_q}a^{l_{q}}$ are reduced and satisfy $w =_G w'$ we have $j = q$ and $\epsilon_i = \delta_i$ for all $1 \leq i \leq j$.  This $j$ now well-defines a length function $L(g) = j$ on $G$, with elements in $\langle a\rangle$ having length zero.

A reduced word $w = a^{k_0}t^{\epsilon_1}a^{k_1}t^{\epsilon_2}a^{k_2}\cdots a^{k_{j-1}}t^{\epsilon_j}a^{k_j}$ is said to be a \emph{normal form} if in addition $ta^{k_i}$ being a subword implies $0\leq k_i<n$ and $t^{-1}a^{k_i}$ being a subword implies $0\leq k_i<m$.  One obtains a normal form from a reduced word $w = a^{k_0}t^{\epsilon_1}a^{k_1}t^{\epsilon_2}a^{k_2}\cdots a^{k_{j-1}}t^{\epsilon_j}a^{k_j}$ by writing $k_j = qm +r$ where $0\leq r<m$, replacing $t^{\epsilon_j}a^{k_j}$ by $a^{qn}t^{\epsilon_j}a^r$ and freely reducing if $\epsilon_j = -1$, or performing the comparable switch in case $\epsilon_j = 1$.  Now we analyze the new $k_{j-1}$ and make the same replacement if possible.  Continuing in this way until one comes to $k_0$ we obtain a normal form.  Each $g\in G$ has a unique normal form $w$ with $w =_G g$.

A word of form $w\equiv a^{k_0}$ or $w \equiv a^{k_0}t^{\epsilon_1}a^{k_1}t^{\epsilon_2}a^{k_2}\cdots a^{k_{j-1}}t^{\epsilon_j}$ is said to be \emph{cyclically reduced} if every cyclic permutation on $w$ is reduced.  Each element of $BS(m, n)$ is conjugate to an element with cyclically reduced normal form.  Let the \emph{cyclic length} of an element $g$ be the minimal length of a cyclically reduced word representing an element conjugate to $g$.  An element is of cyclic length $0$ if and only if it is conjugate to an element of $\langle a\rangle$, and so powers of an element of cyclic length $0$ are also of cyclic length $0$.

\begin{proposition}\label{BSlemma}  Given $m, n \in \mathbb{Z} \setminus \{0\}$ we have the following:

\begin{enumerate}  \item If $p$ is a prime with $p>|m|, |n|$ then $p$-th roots in $BS(m, n)$ of elements in $\langle a\rangle$ must also be elements of $\langle a\rangle$.

\item  If $g\in BS(m, n)$ is of cyclic length at least one then $g$ has unique odd root extraction.
\end{enumerate}

Thus the group $BS(m, n)$ has unique $p$-th root extraction for any prime $p$ satisfying the hypotheses of (1).
\end{proposition}

\begin{proof}  For (1) we suppose that $g =_G w^{-1}a^l w$ is a $p$-th root of $a^k$.  We may assume that $w$ is a reduced word of form $w \equiv t^{\epsilon_1}a^{k_1}t^{\epsilon_2}\cdots t^{\epsilon_j}a^{k_{j}}$ and of shortest length among such words.  We have $(w^{-1}a^lw)^p =_G w^{-1}a^{pl}w$.  Since $w^{-1}a^{pl}w =_G a^k$ we know that $pl \in m\mathbb{Z}$ (in case $\epsilon_1 = 1$) or $pl \in n\mathbb{Z}$ (in case $\epsilon_1 = -1$).  But since $p>|n|, |m|$ we know that $l\in  m\mathbb{Z}$ or $l \in n\mathbb{Z}$.  But now $g =_G w_0 a^{l\frac{m}{n}}w_0$ or $g =_G w_0 a^{l\frac{n}{m}}w_0$ where $w \equiv t^{\epsilon_1}a^{k_1}w_0$, so in fact $w$ must have been of length $0$ and $g\in \langle a \rangle$.

For (2) let $q$ be a positive odd number and let $g$ have positive cyclic length.  By conjugating $g$ if necessary, we may assume without loss of generality that the normal form of $g$ is cyclically reduced.  Then any $q$ root of $g$ must also have a normal form which is cyclically reduced.  We treat cases.  Suppose first that $m = -n$.  Suppose $w\equiv a^{k_0}t^{\epsilon_1}a^{k_1}\cdots a^{\epsilon_{j-1}}t^{\epsilon_j}$ and $w' \equiv a^{k_0'}t^{\epsilon_1}a^{k_1}\cdots a^{\epsilon_{j-1}}t^{\epsilon_j}$ are in normal form and satisfy $w^q =_G g =_G (w')^q$, we let $k_0 = s+ lm$ and $k_0' = s' +l'm$  (that only the $k_0$ and $k_0'$ exponents could possibly differ follows from performing the normal form algorithm to the words $w^q$ and $(w')^q$).  If $j$ is an even number then the normal form of $w^q$ is 

\begin{center} $a^{k_0 + (q-1)lm}(t^{\epsilon_1}a^{k_1}\cdots t^{\epsilon_j}a^{s})^{q-1}(t^{\epsilon_1}a^{k_1}\cdots t^{\epsilon_j})$
\end{center}

\noindent and similarly the normal form of $(w')^q$ is

\begin{center} $a^{k_0' + (q-1)l'm}(t^{\epsilon_1}a^{k_1}\cdots t^{\epsilon_j}a^{s'})^{q-1}(t^{\epsilon_1}a^{k_1}\cdots t^{\epsilon_j})$
\end{center}

\noindent In particular we have $s = s'$ and also $k_0 + (q-1)lm = k_0' + (q-1)l'm$, which implies $qlm = ql'm$, so $l = l'$ and $k_0 = k_0'$.  If $j$ is an odd number then the normal form of $w^q$ is

\begin{center} $a^{k_0}(t^{\epsilon_1}a^{k_1}\cdots t^{\epsilon_j}a^{s})^{q-1}(t^{\epsilon_1}a^{k_1}\cdots t^{\epsilon_j})$
\end{center}

\noindent and similarly the normal form of $(w')^q$ is 

\begin{center} $a^{k_0'}(t^{\epsilon_1}a^{k_1}\cdots t^{\epsilon_j}a^{s'})^{q-1}(t^{\epsilon_1}a^{k_1}\cdots t^{\epsilon_j})$
\end{center}

\noindent so that again $k_0 = k_0'$.

The second case for proving (2) is that wherein $m \neq -n$.  We again suppose that $w\equiv a^{k_0}t^{\epsilon_1}a^{k_1}\cdots a^{k_{j-1}}t^{\epsilon_j}$ and $w' \equiv a^{k_0'}t^{\epsilon_1}a^{k_1}\cdots a^{k_{j-1}}t^{\epsilon_j}$ are in normal form and satisfy $w^q =_G g =_G  (w')^q$.  As

\begin{center} $(a^{k_0}t^{\epsilon_1}a^{k_1}\cdots a^{\epsilon_{j-1}}t^{\epsilon_j})^q =_G (a^{k_0'}t^{\epsilon_1}a^{k_1}\cdots a^{k_{j-1}}t^{\epsilon_j})^q$
\end{center}

\noindent we can cancel on the right by $t^{\epsilon_1}a^{k_1}\cdots a^{\epsilon_{j-1}}t^{\epsilon_j}$ and obtain that 

\begin{center} $(a^{k_0}t^{\epsilon_1}a^{k_1}\cdots a^{k_{j-1}}t^{\epsilon_j})^{q-1}a^{k_0} =_G (a^{k_0'}t^{\epsilon_1}a^{k_1}\cdots a^{k_{j-1}}t^{\epsilon_j})^{q-1}a^{k_0'}$
\end{center}

\noindent Define $\Frac(\epsilon_i)$ to be $\frac{n}{m}$ if $\epsilon_i = 1$ and to be $\frac{m}{n}$ if $\epsilon_i = -1$.  Assume for contradiction that $k_0 \neq k_0'$ and let without loss of generality $k_0 > k_0'$.  Define $k = k_0 - k_0'$.  We know that

\begin{center} $1_G =_G a^{-k_0'} (t^{-\epsilon_j}a^{-k_{j-1}}\cdots a^{-k_1}t^{\epsilon_1}a^{-k_0'})^{q-1}(a^{k_0}t^{\epsilon_1}a^{k_1}\cdots a^{k_{j-1}}t^{\epsilon_j})^{q-1}a^{k_0}$
\end{center}

\noindent and since both $a^{-k_0'} (t^{-\epsilon_j}a^{-k_{j-1}}\cdots a^{-k_1}t^{\epsilon_1}a^{-k_0'})^{q-1}$ and $(a^{k_0}t^{\epsilon_1}a^{k_1}\cdots a^{k_{j-1}}t^{\epsilon_j})^{q-1}a^{k_0}$ are reduced, it must be that $k \in m \mathbb{Z}$ if $\epsilon_1 = 1$ or that $k\in n\mathbb{Z}$ if $\epsilon_1 = -1$.  Thus $t^{-\epsilon_1}a^kt^{\epsilon_1} =_G a^{k\Frac(\epsilon_1)}$.  Reasoning in the same way we see that $k\Frac(\epsilon_1) \in m\mathbb{Z}$ if $\epsilon_2 = 1$ and $k\Frac(\epsilon_1) \in n\mathbb{Z}$ if $\epsilon_2 = -1$.  Continuing in this way we see that 

\begin{center} $1_G =_G  a^{-k_0'} (t^{-\epsilon_j}a^{-k_{j-1}}\cdots a^{-k_1}t^{\epsilon_1}a^{-k_0'})^{q-2}a^{k\Frac(\epsilon_1)\cdots\Frac(\epsilon_j)}(a^{k_0}t^{\epsilon_1}a^{k_1}\cdots a^{k_{j-1}}t^{\epsilon_j})^{q-2}a^{k_0}$
\end{center}

\noindent Letting $T = \Frac(\epsilon_1)\Frac(\epsilon_2)\cdots \Frac(\epsilon_j)$ we perform the above process again to see that 

\begin{center} $1_G =_G  a^{-k_0'} (t^{-\epsilon_j}a^{-k_{j-1}}\cdots a^{-k_1}t^{\epsilon_1}a^{-k_0'})^{q-3}a^{kT^2 + kT}(a^{k_0}t^{\epsilon_1}a^{k_1}\cdots a^{k_{j-1}}t^{\epsilon_j})^{q-3}a^{k_0}$
\end{center}

\noindent Iterating this process $q-3$ more times shows $1_G =_G a^{-k_0'}a^{k(T^{q-1} +T^{q-2} + \cdots + T)} a^{k_0}$.  Thus $0 = k(T^{q-1} +T^{q-2} + \cdots + T +1)$, and since $k \neq 0$ by assumption we have $T^{q-1} +T^{q-2} + \cdots + T +1 = 0$.  The polynomial $x^{q-1} + x^{q-2} + \cdots +x +1$ can only have $1$ or $-1$ as rational roots, but as $q-1$ is even we know that neither can be roots, and so $x^{q-1} + x^{q-2} + \cdots +x +1$ cannot have rational roots.  But $T$ is a root and rational by construction, a contraditcion.
\end{proof}

We note that in the group $BS(1, -1)$ the element $t^2$ has infinitely many $2$-roots.  This can be seen by noticing that $(a^kt)^2 =_G a^kta^kt =_G a^ka^{-k}t^2 =_G t^2$ for any $k\in \mathbb{Z}$ and the $a^kt$ are distinct as elements in $G$ since they are all in normal form.  Other such examples can be contructed when $m = -n$.

\end{section}

\bibliographystyle{amsplain}

\end{document}